\documentclass[12pt]{amsart}

\usepackage{amsfonts}
\usepackage{amssymb}

\theoremstyle{plain}
\newtheorem{theorem}{Theorem}[section]

\newtheorem{lemma}[theorem]{Lemma}

\theoremstyle{definition}

\newcommand{\bC}{\mathbb{C}}

\newcommand{\bN}{\mathbb{N}}
\newcommand{\cG}{\mathcal{G}}
\newcommand{\cP}{\mathcal{P}}
\newcommand{\cH}{\mathcal{H}}
\newcommand{\cS}{\mathcal{S}}

\begin{document}

\baselineskip 6.2mm

\title[On order automorphisms of the effect algebra]{On order automorphisms of the effect algebra\footnote{To appear in Acta Sci. Math. Szeged}} 

\author{Roman Drnov\v sek}


\begin{abstract}
We give short proofs of two \v{S}emrl's descriptions of order automorphisms of the effect algebra. 
This sheds new light on both formulas that look quite complicated. 
Our proofs rely on Moln\'{a}r's characterization of order automorphisms of the cone of all positive operators.
\end {abstract}

\maketitle

 {\it Math. Subj. Classification (2010)}: 47B49 \\

\section{Introduction}

Throughout the paper, let $\cH$ be a complex Hilbert space of  $\dim \cH \ge 2$ with the inner product 
$\langle \cdot , \cdot  \rangle$. 
By $\cS(\cH)$ we denote the set of all bounded linear selfadjoint operators on 
$\cH$. An operator $A \in \cS(\cH)$ is said to be {\it positive}, $A \ge 0$, if 
$\langle A x , x \rangle \ge 0 $ for all $x \in  \cH$, and $A$ is called {\it strictly positive}, $A > 0$, 
if $A$  is positive and invertible.
The set  $\cS(\cH)$ is partially ordered by the relation $\le$ defined by $A \le B \iff B - A \ge 0$.
A real function $f$ defined on an interval $J \subseteq [0, \infty)$ is said to be {\it operator monotone on} $J$ 
if, for every operators $A$ and $B$ in $\cS(\cH)$ with spectra contained in $J$, the inequality  
$A \le B$ implies that $f(A) \le f(B)$. 

An additive map $T : \cH \to \cH$ is {\it conjugate-linear} if $T(\lambda x) = \overline{\lambda} T x$
for every $x \in \cH$ and $\lambda \in \bC$. For a bounded conjugate-linear operator $T : \cH \to \cH$
we define the adjoint $T^*$ to be the unique bounded conjugate-linear map $T^* : \cH \to \cH$
satisfying $\langle T x , y \rangle = \overline{\langle x , T^* y \rangle}$
for all pairs $x, y \in \cH$.

Let us recall the definition of the group $G_1$ from \cite{Se17b}.
By $GL(\cH)$ we denote the general linear group on $\cH$,
that is, the multiplicative group of all invertible linear bounded operators on
$\cH$. Furthermore, let $CGL(\cH)$ denote the group of all invertible bounded
either linear or conjugate-linear operators on $\cH$. The multiplicative group 
$S^1 = \{z \in  \bC : |z| = 1\}$  can be naturally embedded  into CGL(H) by $z \mapsto z I$, 
where $I$ denotes the identity operator on $\cH$. By $G_1$  we denote the quotient
group  $G_1 = CGL(H)/S^1$. 
Note that each element of $G_1$ has the form $[T] = \{ z T : z \in S^1\}$, where $T : \cH \to
\cH$ is a bounded invertible either linear or conjugate-linear operator on $\cH$.

Assume that $A$ and $B$ are operators in $\cS(\cH)$ such that $B - A > 0$. 
Then {\it operator intervals} are defined by 
$$ [A, B] = \{ C \in \cS(\cH) : A  \le C \le B\} , $$
$$ (A, B] = \{ C \in \cS(\cH) : A  < C \le B\} , $$
and
$$ (A, B) = \{ C \in \cS(\cH) : A  < C < B\}  . $$
Similarly, we define
$$ (A, \infty) = \{ C \in \cS(\cH) : A  <  C \}  , $$
$$ [A, \infty) = \{ C \in \cS(\cH) : A  \le  C \}  , $$
and  $(-\infty, \infty) = \cS(\cH)$. 
The notations  $[A, B)$, $(-\infty, A)$, $(-\infty, A]$ are now self-explanatory.
The operator interval $[0, I]$ is also called the {\it effect algebra} on $\cH$.

Let $J$ and $K$ be operator intervals. A bijective map 
$\phi : J \to K$ is called an {\it order isomorphism} if for every pair of operators 
$A$, $B$ in $J$ we have
$$ A \le B \iff \phi(A) \le \phi(B) ,$$
and it is called an {\it order anti-isomorphism} if for every pair of operators 
$A$, $B$ in $J$, 
$$ A \le B \iff \phi(A) \ge \phi(B) .$$
If, in addition, $J = K$, the map $\phi$ is called an {\it order automorphism} in the first case, and 
an {\it order anti-automorphism} in the second one.

Order isomorphisms of operator intervals have been studied systematically in \cite{Se17a}.
Their study was motivated by problems in mathematical physics; see \cite{Mo01}, \cite{Mo07} and the references therein.
In particular, it was shown in \cite{Se17a} that every operator interval is order isomorphic or order 
anti-isomorphic to one of the following operator intervals: $(-\infty, \infty)$,  $[0, \infty)$, $(0, \infty)$, 
and  $[0, I]$.
In the first three cases order automorphisms have simple forms. For example, order automorphisms of the operator interval  
$[0, \infty)$ are only congruence transformations 
$A \mapsto S A S^*$ for some invertible bounded either linear or conjugate-linear operator  
$S : \cH \to \cH$ (see Theorem \ref{auto_positive}).
The form of order automorphisms looks quite complicated only in the case of the effect algebra  $[0, I]$.
Namely, it was proved in \cite{Se17a} that each order automorphism $\phi$ of the effect algebra $[0, I]$ has the form 
$$ \phi (A) = f_p \left( (I+(T T^*)^{-1})^{1/2} (I - (I + T A T^*)^{-1}) (I+(T T^*)^{-1})^{1/2} \right) , 
\ A \in [0, I] , $$
where $T : \cH \to \cH$ an invertible bounded either linear or conjugate-linear operator, $p$ is a negative real number, 
and $f_p$ is the operator monotone function defined by \eqref{function f_p} below. 
The aim of the present paper is to give a short proof that better explains this formula. \\

\section{Preliminaries}

For every real number $p < 1$ the function $f_p$ is defined by 
\begin{equation}
\label{function f_p}
 f_p (x) = \frac{x}{p x + 1- p} .
\end{equation}
For  $p \in [0,1)$, let the interval  $[0, \infty)$ be the domain of $f_p$, while for $p < 0$ its domain is the interval $[0, 1-\frac{1}{p})$. Clearly, the common domain for all members of the set $\cG= \{f_p : p < 1\}$ 
is the interval $[0, 1]$, and each function $f_p \in \cG$ is a bijective increasing function on the unit interval [0,1] onto itself. By $\cP$ we denote the multiplicative group of positive real numbers.
We begin with a lemma that slightly improves \cite[Lemma 3.10]{Se17a}.

\begin{lemma}
\label{f_p-monotone}
For  every $p \in [0,1)$ the function $f_p$ is operator monotone on the interval $[0, \infty)$, and for every real number 
$p < 0$  it is operator monotone on the interval $[0, 1-\frac{1}{p})$.
The set $\cG$ with the operation of functional composition is a group of functions from $[0, 1]$ onto itself, and we have 
$$ f_p \circ f_q = f_{p+q-pq} $$
and 
\begin{equation}
\label{inverse}
 f_p^{-1} =  f_{\frac{p}{p-1}} .
\end{equation}
Moreover, the map $p \mapsto f_{1-p}$ is a group isomorphism between the groups $\cP$ and $\cG$. 
\end{lemma}

\begin{proof}
Since the function $f_0(x) = x$ is clearly operator monotone on the interval $[0, \infty)$,
we can assume that $p \neq 0$. 
We will use the well-known fact that if  $A$ and $B$ in $\cS(\cH)$ are strictly positive operators,
then $A \le B \iff A^{-1} \ge B^{-1}$. 
If $p \in (0,1)$, then for any positive operator $A$, 
$$ f_p(A) = \frac{1}{p} \, I - \frac{1-p}{p^2} \left( A +  \left( \frac{1}{p}-1 \right) I \right)^{-1} , $$
and so the function $f_p$ is operator monotone on the interval $[0, \infty)$.
If $p < 0$, then for any operator $A \in [0, (1-1/p) I )$, 
$$ f_p(A) = \frac{1}{p} \, I + \frac{1-p}{p^2} \left( \left( 1-\frac{1}{p} \right) I - A\right)^{-1} , $$
implying that the function $f_p$  is operator monotone on the interval $[0, 1-\frac{1}{p})$.

It is straightforward to verify the remaining assertions.
\end{proof}

The following assertion is actually shown in the proof of \cite[Corollary 5.2]{Se17a}. 

\begin{lemma}
\label{restricted} 
If $\phi : [0,I] \to  [0,I]$ is an order automorphism, then 
$$ \phi((0,I]) = (0,I]. $$
\end{lemma}

The following description of order automorphisms of the operator interval  $[0, \infty)$ was proved in \cite{Mo01}; see also \cite[Theorem 2.5.1]{Mo07}.

\begin{theorem}
\label{auto_positive}
Assume that $\phi : [0, \infty) \to [0, \infty)$ is an order
automorphism. Then there exists an invertible bounded either linear or conjugate-linear operator  
$S : \cH \to \cH$ such that 
$$ \phi(A) = S A S^* $$
for every $A \in [0, \infty)$.
\end{theorem}

\vspace{3mm}
\section{Results}

The functions from the group $\cG$ induce order automorphisms of the effect algebra $[0, I]$
via functional calculus.
 
\begin{lemma}
\label{homeo}
For every real number $p < 1$, the map $A \mapsto f_p (A)$, $A \in [0, I]$, is an order automorphism 
and a homeomorphism of the effect algebra $[0, I]$.
\end{lemma}

\begin{proof}
Since the functions $f_p$ and $f_p^{-1}$ are operator monotone on the interval $[0,1]$ by Lemma \ref{f_p-monotone}, 
the map $A \mapsto f_p (A)$ is an order automorphism of the effect algebra $[0, I]$.
Furthermore, the two formulas for $f_p (A)$ given in the proof of Lemma \ref{f_p-monotone} ensure that 
the map $A \mapsto f_p (A)$  is a composition of translations, multiplication by a real constant and the map 
$A \mapsto A^{-1}$ (defined on the set of all invertible operators). 
Since all of these maps are continuous, we conclude together with \eqref{inverse} that 
the map $A \mapsto f_p (A)$ is continuous in both directions. This completes the proof.
\end{proof}

The following order automorphisms of the effect algebra $[0, I]$ were introduced in \cite{Se17a}.

\begin{lemma}
\label{fi_p,T}
Let $p$ be a negative real number and $T : \cH \to \cH$ an invertible bounded either linear or conjugate-linear operator. Then  the map $\phi_{p, T} : [0,I] \to  [0,I]$ given by 
$$ \phi_{p, T} (A) = f_p \left( (I+(T T^*)^{-1})^{1/2} (I - (I + T A T^*)^{-1}) (I+(T T^*)^{-1})^{1/2} \right) , 
\ A \in [0, I] , $$
is an order automorphism and a homeomorphism of the effect algebra $[0, I]$. 
\end{lemma}

\begin{proof}
Since the map $A \mapsto I - (I + T A T^*)^{-1}$, \ $A \in [0, I]$, is an order isomorphism of  the effect algebra 
$[0, I]$ onto the operator interval $[0,  (I+(T T^*)^{-1})^{-1}]$, the map 
$$ A \mapsto  (I+(T T^*)^{-1})^{1/2} (I - (I + T A T^*)^{-1}) (I+(T T^*)^{-1})^{1/2} = 
f_p^{-1}( \phi_{p, T} (A))  ,  \ A \in [0, I] , $$
is an order automorphism of the effect algebra $[0, I]$.
Now, we apply Lemma \ref{homeo} to conclude that $\phi_{p, T}$  is also an order automorphism of the effect algebra $[0, I]$.  

To prove the continuity of $\phi_{p, T}$, we note that 
$$ \phi_{p, T} (A) =
f_p \left( \frac{1}{2} \, (I+(T T^*)^{-1})^{1/2} f_{1/2}(T A T^*) (I+(T T^*)^{-1})^{1/2} \right) , 
\ A \in [0, I] .$$
Hence, $\phi_{p, T}$ is a composition of two congruence transformations and two maps from the Lemma \ref{homeo}. Since all of them are continuous, it is continuous as well.
The same conclusion holds for the inverse of  $\phi_{p, T}$, completing the proof.
\end{proof}

As the main contribution of this paper, we give a short proof of \cite[Theorem 2.3]{Se17a}.
Our proof relies on Theorem \ref{auto_positive}, and it gives us a better insight into the structure of 
order automorphisms of the effect algebra $[0, I]$.

\begin{theorem}
\label{short_proof}
Assume that $\phi : [0, I] \to [0, I]$ is an order
automorphism. Then there exist a negative real number $p$ and an invertible
bounded either linear or conjugate-linear operator  $T : \cH \to \cH$ such that 
$\phi (A) = \phi_{p, T} (A)$ for all $A \in [0, I]$.
\end{theorem}

\begin{proof}
By Lemma \ref{restricted}, the restriction $\phi|_{(0,I]}$ is 
an order automorphism on the operator interval $(0,I]$.   
The map $A \mapsto A^{-1} - I$ is an order anti-isomorphism of $(0, I]$ onto  $[0, \infty)$ 
and its inverse  $A \mapsto (I+A)^{-1}$ is an order anti-isomorphism of $[0, \infty)$ onto $(0, I]$.
These observations together with Theorem \ref{auto_positive} imply that 
there exists an invertible bounded either linear or conjugate-linear operator  
$S : \cH \to \cH$ such that  
$$ \phi(A) = (I + S (A^{-1} - I)S^*)^{-1}$$
for all $A \in (0, I]$. Choose a real number $\lambda \in (1, \infty)$ such that $\lambda > \|S S^*\| = \|S\|^2$.  
Let $A \in (0, I] $. Then 
$$ \phi(A)^{-1} + (\lambda-1) I=  \lambda I - S S^* + S A^{-1}S^* = 
(\lambda I-S S^*)^{1/2} (I + R A^{-1} R^*) (\lambda I-S S^*)^{1/2} , $$
where an invertible bounded either linear or conjugate-linear operator  
$R : \cH \to \cH$ is given by  
$$ R = (\lambda I-S S^*)^{-1/2} \, S .  $$
The equality $R^* = T^{-1}$ defines an invertible bounded either linear or conjugate-linear operator  
$T : \cH \to \cH$. Clearly, we have 
$$ (T T^*)^{-1} = R R^*  =  (\lambda I-S S^*)^{-1} \, S S^* , $$
so that
$$ I+(T T^*)^{-1} = \lambda  (\lambda I-S S^*)^{-1} $$
and 
$$ \lambda I-S S^* = \lambda (I+ (T T^*)^{-1})^{-1} . $$
Therefore,  
$$ \phi(A)^{-1} + (\lambda-1) I= 
(\lambda I-S S^*)^{1/2} (I + (T  A T^*)^{-1}) (\lambda I-S S^*)^{1/2} = $$
$$ = \lambda (I+(T T^*)^{-1})^{-1/2}(I + T  A T^*) (T  A T^*)^{-1}  (I+(T T^*)^{-1})^{-1/2} . $$
It follows that 
$$ \lambda (\phi(A)^{-1} + (\lambda-1) I)^{-1}= 
(I+(T T^*)^{-1})^{1/2}(T  A T^*)(I + T  A T^*)^{-1}  (I+(T T^*)^{-1})^{1/2} , $$
 and so 
$$ f_q( \phi(A)) = (I+(T T^*)^{-1})^{1/2}(I- (I + T  A T^*)^{-1})  (I+(T T^*)^{-1})^{1/2} , $$
where $q = 1-1/\lambda \in (0,1)$. 
Letting $p= q/(q-1) \in (-\infty, 0)$  we have $f_q = f_p^{-1}$ by \eqref{inverse}, 
and so we obtain that $\phi(A) =  \phi_{p, T} (A)$  for every $A \in (0, I]$.

It remains to prove that  $\phi(A) =  \phi_{p, T} (A)$  for every $A \in [0, I]$. 
In other words, we need to show that $\phi_{p, T}$ is the only order automorphism on the effect algebra $[0, I]$
that extends the restriction $\phi|_{(0,I]}$.
Given $A \in [0, I]$, define a decreasing sequence $\{A_n\}_{n \in \bN}$ in the order interval $(0, I]$ by 
$A_n = (1- \frac{1}{n}) A + \frac{1}{n} I$. Clearly, it converges to $A$, and so 
Lemma \ref{fi_p,T} implies that  $\{\phi_{p, T}(A_n)\} = \{\phi(A_n)\}$ converges to $B = \phi_{p, T} (A)$.
Since the sequence $\{\phi(A_n)\}$ is decreasing, it is easily seen that $\phi(A_n) \ge B$ for all $n$. 
From $A_n \ge A$ it follows that  $\phi(A_n) \ge \phi(A)$ for all $n$, and so $B \ge \phi(A)$. 
Since $\phi$ is bijective, there is an operator $A^\prime \in [0, I]$ such that $\phi(A^\prime) = B$. Hence, 
$A^\prime \ge A$, as $\phi(A^\prime) \ge \phi(A)$.  
Since  $\phi(A_n) \ge \phi(A^\prime)$,  we have $A_n \ge A^\prime$ for all $n$, and so $A  \ge A^\prime$.
Thus, $A  = A^\prime$ and  $\phi(A) = B$. So,   $\phi(A) =\phi_{p, T} (A)$  for every $A \in [0, I]$. 
\end{proof}

The preceding proof also reveals that the group of order automorphisms of the effect algebra $[0, I]$ is isomorphic to 
the group $G_1$, as it has been shown in \cite{Se17b}.

In \cite{Se12a} another description of order automorphisms of the effect algebra $[0, I]$ was given. 
We now show that it is closely related to the description from Theorem \ref{short_proof}.

\begin{theorem}
\label{another}
Assume that $\phi : [0, I] \to [0, I]$ is an order
automorphism. Then there exist a negative real number $p$, a real number $r \in (0,1)$  and an invertible
bounded either linear or conjugate-linear operator  $S : \cH \to \cH$ with $\| S \| \le 1$ such that 
$$ \phi (A) =  f_p \left( (f_r(S S^*))^{-1/2} f_r(S A S^*) (f_r(S S^*))^{-1/2} \right) $$
for all $A \in [0, I]$.
\end{theorem}

\begin{proof}
By Theorem \ref{short_proof}, there exist a negative real number $p$ and an invertible
bounded either linear or conjugate-linear operator  $T : \cH \to \cH$ such that 
$\phi (A) = \phi_{p, T} (A)$ for all $A \in [0, I]$. If $\|T\| \le 1$ then take $S = T$ and observe that 
$$  f_p^{-1}( \phi_{p, T} (A)) = (f_{1/2}(S S^*))^{-1/2} f_{1/2}(S A S^*) (f_{1/2}(S S^*))^{-1/2} , $$
so that the desired formula with $r=1/2$ follows.

Assume therefore that $\|T\| > 1$. Let
$$ S = \frac{1}{\|T\|} \, T \ \  \textrm{ and } \ \  r = \frac{\|T\|^2}{1+\|T\|^2} , $$
so that 
$$ \|S\| = 1 \ \  \textrm{ and } \ \  \frac{r}{1-r}  = \|T\|^2 . $$
Then 
$$ (f_r(S S^*))^{-1} = r + (1-r)(S S^*)^{-1} = r \, (I+(T T^*)^{-1}) , $$
and
$$ r \, f_r(S A S^*) =   I - (1-r) (r S A S^*+(1-r)I )^{-1}  = I - (I + T A T^*)^{-1} , $$
and so  
$$  (f_r(S S^*))^{-1/2} f_r(S A S^*) (f_r(S S^*))^{-1/2} =  $$
$$ = (I+(T T^*)^{-1})^{1/2} (I - (I + T A T^*)^{-1}) (I+(T T^*)^{-1})^{1/2} = f_p^{-1}( \phi_{p, T}(A) )  , $$
completing the proof.
\end{proof}

\vspace{3mm}
{\it Acknowledgment.} The author acknowledges the financial support from the Slovenian Research Agency  (research core funding No. P1-0222). 
He is also grateful to the referee for a useful remark that shortened the proof of Lemma \ref{homeo}.

\vspace{2mm}

\baselineskip 6mm
\noindent
Roman Drnov\v sek \\
Department of Mathematics \\
Faculty of Mathematics and Physics \\
University of Ljubljana \\
Jadranska 19 \\
SI-1000 Ljubljana, Slovenia \\
e-mail : roman.drnovsek@fmf.uni-lj.si 

\end{document}